\documentclass[10pt]{amsart} 
 \newcommand{\filename}{{dynamical-coord-size-29-Jan-2025.tex}} 
\usepackage{enumitem}
\usepackage{amsmath,amsthm,amssymb,mathrsfs,xypic,xcolor}

\renewcommand{\geq}{\geqslant}
\renewcommand{\leq}{\leqslant}
\newcommand{\Osh}{{\mathcal O}}                        

\newcommand{\K}{\mathrm{K}}      
                      
\newcommand{\Ish}{\mathcal{I}}

\newcommand{\Vol}{\operatorname{Vol}}

\newcommand{\ord}{\mathrm{ord}}

\newcommand{\KK}{\mathbf{K}}
\newcommand{\FF}{\mathbf{F}}
\newcommand{\PP}{\mathbb{P}} 
\newcommand{\QQ}{\mathbb{Q}} 
\newcommand{\RR}{\mathbb{R}} 

\newtheorem{theorem}{Theorem}[section]

\newtheorem{corollary}[theorem]{Corollary}
\newtheorem{proposition}[theorem]{Proposition}

\theoremstyle{definition}
\newtheorem{defn}[theorem]{Definition}

\newtheorem{remark}[theorem]{Remark}

\newtheorem{example}[theorem]{Example}

\newtheorem{conjecture}[theorem]{Conjecture}

\newtheorem{problem}[theorem]{Problem}

\numberwithin{equation}{section}

\begin{document}

\author{Nathan Grieve}
\address{Department of Mathematics \& Statistics,
Acadia University, Huggins Science Hall, Room 130,
12 University Avenue,
Wolfville, NS, B4P 2R6
Canada; 
School of Mathematics and Statistics, 4302 Herzberg Laboratories, Carleton University, 1125 Colonel By Drive, Ottawa, ON, K1S 5B6, Canada; 
D\'{e}partement de math\'{e}matiques, Universit\'{e} du Qu\'{e}bec \`a Montr\'{e}al, Local PK-5151, 201 Avenue du Pr\'{e}sident-Kennedy, Montr\'{e}al, QC, H2X 3Y7, Canada;
Department of Pure Mathematics, University of Waterloo, 200 University Avenue West, Waterloo, ON, N2L 3G1, Canada
}

\email{nathan.m.grieve@gmail.com}%

\author{Chatchai Noytaptim}

\address{
Department of Pure Mathematics, University of Waterloo, 200 University Avenue West, Waterloo, ON, N2L 3G1, Canada
}

\email{cnoytaptim@uwaterloo.ca}

\thanks{
The first author thanks the Natural Sciences and Engineering Research Council of Canada for their support through his grants DGECR-2021-00218 and RGPIN-2021-03821. \\
\\
Date: \today.  \\
File name: \filename
}

\title[Vojta's height inequalities and asymptotic coordinate size dynamics]{On relative fields of definition for log pairs, Vojta's height inequalities and asymptotic coordinate size dynamics}

\begin{abstract}
We build on the perspective of the works \cite{Grieve:Noytaptim:fwd:orbits}, \cite{Matsuzawa:2023}, \cite{Grieve:qualitative:subspace}, \cite{Grieve:chow:approx}, \cite{Grieve:Divisorial:Instab:Vojta} (and others) and study the dynamical arithmetic complexity of rational points in projective varieties.  Our main results make progress towards the attractive problem of asymptotic complexity of coordinate size dynamics in the sense formulated by Matsuzawa, in  \cite[Question 1.1.2]{Matsuzawa:2023}, and  building on earlier work  of Silverman \cite{Silverman:1993}.  A key tool to our approach here is a novel formulation of conjectural Vojta type inequalities for log canonical pairs and with respect to finite extensions of number fields.  Among other features, these conjectured Diophantine arithmetic height inequalities raise the question of existence of log resolutions with respect to finite extensions of number fields which is another novel concept which we formulate in precise terms here and also which is of an independent interest.
\end{abstract}

\maketitle

\section{Introduction}\label{intro}

Our purpose here, is to study the dynamical arithmetic complexity of rational points in projective varieties.  In doing so, we build on the perspective of the recent works \cite{Grieve:Noytaptim:fwd:orbits}, \cite{Matsuzawa:2023}, \cite{Grieve:qualitative:subspace}, \cite{Grieve:chow:approx}, \cite{Grieve:Divisorial:Instab:Vojta} and others.  Our main results, see Theorems \ref{coord:sizes:main:thm}, \ref{coord:sizes:main:thm:cor} and \ref{coord:sizes:ZD:main:thm:cor}, make progress towards the attractive Problem \ref{main:problem} below.  Problem \ref{main:problem} encompasses \cite[Question 1.1.2]{Matsuzawa:2023} as a special case (compare also with \cite[Theorem 1.16]{Matsuzawa:2023} and \cite[Theorem E]{Silverman:1993}).  

Throughout this article, we work over a base number field $\KK$  and all finite extension fields $\FF / \KK$ are contained in $\overline{\KK}$ a fixed choice of algebraic closure of $\KK$.  We refer to Section \ref{abs:weil:function:conventions} for details about our conventions in regards to absolute values, local Weil and height functions.  They follow the approach from \cite{Grieve:qualitative:subspace}. Unless stated otherwise, all varieties are assumed to be defined over $\KK$ and are assumed to be geometrically integral.

\begin{problem}[Compare with {\cite[Question 1.1.2]{Matsuzawa:2023}}]\label{main:problem} 
Let $X$ be a projective variety over $\KK$.
Let $D'$ be a nonzero and  effective $\QQ$-Cartier divisor over $X$ and defined over some finite extension field $\FF / \KK$.  Fix a model $\mu' \colon X' \rightarrow X$ 
with $D'$ supported on $X'$ as a $\QQ$-Cartier divisor defined over $\FF$. Let $L$ be a big line bundle on $X$ and defined over $\KK$.  Let $f \colon X \rightarrow X$ 
be a surjective morphism and denote its $n$th iterate by
$ f^{(n)} \colon X \rightarrow X \text{.} $
Let $S \subset M_{\KK}$ be a finite set of places.  Let $x \in X(\KK)$.   We assume that 
$f^{(n)}(x) \not \in \operatorname{Center}_X(D')$ 
for $n \gg 0$ and for each such $n \gg 0$ we set $x_n' := (\mu ')^{-1}(f^{(n)}(x)) \text{.}$ 
Then, in this context, under what conditions on $f$, $D'$ and $x$ is it true that 
$$
\lim_{n \to \infty} \frac{ \sum_{v \in S} \lambda_{D'}(x_n', v)}{h_L(f^{(n)}(x))} = 0 \text{?}
$$
\end{problem}

Example \ref{eg:main:problem} below explains how Problem \ref{main:problem} includes, as a special case, \cite[Question 1.1.2]{Matsuzawa:2023}.

\begin{example}\label{eg:main:problem}
Consider the case of a subscheme $Y \subseteq X$, for $X$ a projective variety over $\KK$, and with $Y$ defined over a finite extension field $\FF / \KK$.  Let 
$$\pi \colon X' = \operatorname{Bl}_Y(X_{\FF}) \rightarrow X_{\FF}$$ 
be the blowing-up of $X$ along $Y$.  Let $E$ be the exceptional divisor.  Then the local Weil functions $\lambda_Y(\cdot,v)$ for $v \in M_{\KK}$ can be described as
$$
\lambda_Y( \cdot, v) := \lambda_E( \pi ^{-1}(\cdot), v) \text{.}
$$

Let $f \colon X \rightarrow X$ 
be a surjective morphism with $n$th iterate $ f^{(n)} \colon X \rightarrow X \text{.}$  Then, fixing $x \in X_{\FF}(\KK) \setminus Y \text{,}$ 
setting $x_n' := \pi ^{-1}(f^{(n)}(x))$
and assuming that $f^{(n)}(x) \not \in Y$
for all $n \gg 0$, in this context, Problem \ref{main:problem} asks the extent to which it holds true that 
$$
\lim_{n \to \infty} \frac{ \sum_{v \in S} \lambda_Y(f^{(n)}(x), v)  }{ h_L(f^{(n)}(x)) } = \lim_{n \to \infty}  \frac{ \sum_{v \in S} \lambda_E(x_n', v)  }{ h_L(f^{(n)}(x)) } = 0 \text{.}
$$
Problem \ref{main:problem} is thus a more general formulation of \cite[Question 1.1.2]{Matsuzawa:2023} (compare also with \cite[Theorem E]{Silverman:1993}).
\end{example}

Our approach towards Problem \ref{main:problem} builds on the strategy employed in \cite{Matsuzawa:2023}.  However, here we work in a much more general context and deduce our main results (Theorems \ref{coord:sizes:main:thm}, \ref{coord:sizes:main:thm:cor} and \ref{coord:sizes:ZD:main:thm:cor}) as a consequence of the following general and novel formulation of a Vojta type conjecture for log canonical pairs.

\begin{conjecture}\label{Vojta:Conj:lc:pairs:intro}
Let $X$ be a projective variety over $\KK$ and having canonical singularities. 
Let $D'$ be a nonzero effective $\QQ$-Cartier divisor over $X$, defined over a finite extension field $\FF / \KK$  and supported as a Cartier divisor defined over $\FF$ on some proper model $\mu' \colon X' \rightarrow X$ 
of $X$, defined over $\KK$.    Assume that $(X',D')$ is log canonical.  Fix a finite set of places $S \subset M_{\KK}$.  Let $\epsilon' > 0$ and $L' = (\mu')^* L$ 
for $L$ a big line bundle on $X$.  Then there exists a proper Zariski closed subset $Z \subsetneq X$ which is defined over $\KK$ and which is such that 
$$
\sum_{v \in S} \lambda_{D'}(\cdot,v) + h_{(\mu')^*\K_X} (\cdot) \leq \epsilon' h_{L'}(\cdot) + \mathrm{O}(1) \text{ over $X \setminus Z$.}
$$
\end{conjecture}

Again, we refer to Section \ref{abs:weil:function:conventions} for details as to our conventions for the local Weil functions $\lambda_{D'}(\cdot,v)$, which are normalized with respect to $\KK$ and which allow the possibility that $D'$ is defined over a finite extension field $\FF / \KK$.  This follows the approach of \cite{Grieve:qualitative:subspace}.  The Diophantine arithmetic inequality, which we have formulated in Conjecture \ref{Vojta:Conj:lc:pairs:intro}, is of a similar flavour but goes beyond what has previously been proposed along these lines (see for instance \cite{Yasuda:2018}, \cite{Matsuzawa:2023} and the references therein).  In particular, Conjecture \ref{Vojta:Conj:lc:pairs:intro} is of an independent interest.

Another interesting question which arises in our perspective here is the notion of $\FF/\KK$ log resolution.  This is formulated in Section \ref{birational:geom:conditions}.  The main motivation for this concept is the possibility of applying a more general form of Vojta's Conjecture, Conjecture \ref{pair:main:conj:intro} below, on a suitable $\FF / \KK$ log resolution to deduce the conclusion of Conjecture \ref{Vojta:Conj:lc:pairs:intro}.

\begin{conjecture}[Compare for example with {\cite[Conjecture 14.3.2]{Bombieri:Gubler}}]\label{pair:main:conj:intro}
Working over a base number field $\KK$, let $S$ be a finite set of places.  Let $X$ be a nonsingular complete variety over $\KK$.  Let $D$ be a normal crossing divisor on $X$ and defined over a finite extension field $\FF / \KK$.    Let $L$ be a big divisor on $X$ and defined over $\KK$.  Let $\epsilon>0$.  Then there exists a proper Zariski closed subset $Z \subsetneq X$ which is such that the inequality
$$
\sum_{v \in S} \lambda_{D}(x,v) + h_{\K_X}(x) \leq \epsilon h_L(x) + \mathrm{O}(1)
$$
is valid for all $x \in X \setminus Z$.
\end{conjecture}

Conjecture \ref{pair:main:conj:intro} allows the possibility of the normal crossings divisor to be defined over a finite extension of the base number field $\KK$.  Further, by passing to an \emph{$\FF / \KK$ log resolution}, in the sense that we define in Section \ref{birational:geom:conditions}, Conjecture \ref{pair:main:conj:intro} allows the possibility to deduce forms of Conjecture \ref{Vojta:Conj:lc:pairs:intro} from Conjecture \ref{pair:main:conj:intro}.  (See Proposition \ref{Vojta:log:res:implies:Vojta:lc:pairs}.)  As explained in Example \ref{Yasuda:Eg:Conj}, Conjecture \ref{Vojta:Conj:lc:pairs:intro} itself is of a similar flavour to \cite[Conjecture 5.2]{Yasuda:2018}.  

These more general forms of Vojta's Conjecture (Conjectures \ref{Vojta:Conj:lc:pairs:intro} and \ref{pair:main:conj:intro}), which we formulate here, are very natural from the vantage point of Schmidt's Subspace Theorem.  Indeed, as explained in Example \ref{Yasuda:Eg:Conj}, Schmidt's Subspace Theorem provides the most basic known instances of the conclusion of these conjectured inequalities.  Further examples of the validity of these conjectured inequalities, in the context of certain $\K$-unstable $\QQ$-Fano varieties with canonical singularities, arise in light of the work \cite{Grieve:Divisorial:Instab:Vojta}.  (See also Example \ref{Ding:destab:eg}.)

The question of existence of $\FF/\KK$ log resolutions, quite generally, in the manner that we formulate in Section \ref{birational:geom:conditions}, appears to be an  unsolved question that is of an independent interest and does not appear to follow directly from the by now classical approach of Hironaka to the existence of log resolutions (including \cite{Hironaka:1964}, \cite{Kollar:Mori:1998} or \cite{Bierstone-et-al}).  In either case, of course Conjecture \ref{pair:main:conj:intro} implies Conjecture \ref{Vojta:Conj:lc:pairs:intro} when $\FF = \KK$.  

Here we deduce our main results (Theorems \ref{coord:sizes:main:thm}, \ref{coord:sizes:main:thm:cor} and \ref{coord:sizes:ZD:main:thm:cor}) as a consequence of Conjecture \ref{Vojta:Conj:lc:pairs:intro}.  It turns out that Theorem \ref{coord:sizes:main:thm} which is proved as an application of Conjecture \ref{Vojta:Conj:lc:pairs:intro}, implies Theorem \ref{coord:sizes:main:thm:cor} below.  This result, Theorem \ref{coord:sizes:main:thm:cor} below, extends \cite[Theorem 1.16]{Matsuzawa:2023} to the case of big line bundles on varieties with canonical singularities and nonzero effective divisors with coefficients in some finite extension of the base number field $\KK$.

\begin{theorem}\label{coord:sizes:main:thm:cor}  Fix a finite set of places $S \subset M_{\KK}$.
Let $f \colon X \rightarrow X$ be a surjective morphism, defined over $\KK$, for $X$ a projective variety over $\KK$ and having canonical singularities.  Let $D$ be a nonzero effective Cartier divisor on $X$ and defined over  some finite extension $\FF / \KK$.  Let $L$ be a big line bundle on $X$ and defined over $\KK$.  Let $x \in X(\KK)$ and let $\alpha_f(x)$ denote the arithmetic degree of $x$ with respect to $f$. 
\begin{enumerate}
\item
Assume that
\begin{itemize}
\item $\alpha_f(x) > 1$; and
\item
 $e_f(D) < \alpha_f(x)$ where 
$$e_f(D) := \left( \liminf_{n \to \infty} \left( \inf_E \frac{1}{\operatorname{ord}_E((f^{(n)})^* D)} \right)^{1/n}  \right)^{-1} \text{.}$$ 
\end{itemize}
Fix $e \in \QQ_{>0}$ with $e_f(D) \leq e < \alpha_f(x)$, fix $\epsilon_0 \in \RR_{>0}$, $\epsilon \in \QQ_{>0}$ and a positive integer $m_0$ which is such that 
\begin{itemize}
\item[(i)] $e + \epsilon < \alpha_f(x)$; and
\item[(ii)] $(e + \epsilon)^{m_0} < \alpha_f(x)^{m_0} \epsilon_0$.
\end{itemize}
Then the pair $\left(X, \frac{1}{(e + \epsilon)^{m_0}} f^{(m_0)*} D \right)$ is log canonical. 
\item
Further assume, in addition to the hypothesis of item (1), that
\begin{itemize}
\item the forward orbit $\mathrm{O}_f(x)$ is \emph{generic} in the sense that $\mathrm{O}_f(x)$ has finite intersection with all proper Zariski closed subsets of $X$;  and 
\item that the conclusion of Conjecture \ref{Vojta:Conj:lc:pairs:intro} applied to the log canonical pair $\left(X,\frac{1}{(e+\epsilon)^{m_0}}(f^{(m_0)})^*D\right)$ and $\epsilon' = 1$ holds true.
\end{itemize}
Then
$$
\lim_{n \to \infty} \frac{ \sum_{v \in S} \lambda_{D}(f^{(n)}(x),v)}{h_L(f^{(n)}(x))} = 0 \text{.}
$$
\end{enumerate}
\end{theorem}

We prove Theorem \ref{coord:sizes:main:thm:cor} in Section \ref{main:thms:proof} as an application of Theorem \ref{coord:sizes:main:thm}.  

\begin{example}\label{hypothesis:example}
To place the hypothesis of item (1), in Theorem \ref{coord:sizes:main:thm:cor}, into its proper context, we note that we can take $e = 1$ if the iterates of $f$ are not ramified along $D$ in the sense that $\operatorname{ord}_E((f^{(n)})^*D) = 1$, if not zero, for all sufficiently large $n$ and all $E$.    In other words, the iterates of $D$ should not be ramified at the centres of $E$ on $X$.  To place this condition that the iterates of $f$ are not ramified along $D$ into perspective, note that aside from the situation that $f$ is an isomorphism, we also recall that for the case that $X$ is a $Q$-abelian variety with canonical singularities, in the sense of \cite[p.~ 246]{Meng:Zhang:2018}, that each surjective self map is \'etale in codimension $1$ as noted in \cite[Lemma 2.1]{Oguiso:2024}.  To place the hypothesis of item (2), in Theorem \ref{coord:sizes:main:thm:cor}, into perspective we refer to Section \ref{vojta:lc:pairs:sec} for situations in which the conclusion of Conjecture \ref{Vojta:Conj:lc:pairs:intro} holds true.  
\end{example}

To provide context and motivation for Theorem \ref{coord:sizes:ZD:main:thm:cor} below, let us first highlight, in Theorem \ref{Grieve:Noytaptim:fwd:orbits:main:thm}, the following very interesting and novel result from  \cite{Grieve:Noytaptim:fwd:orbits}.

\begin{theorem}[{\cite[Theorem 1.2]{Grieve:Noytaptim:fwd:orbits}}]\label{Grieve:Noytaptim:fwd:orbits:main:thm}
Working over a base number field $\KK$, let $X$ be a projective variety.  Let $f \colon X \rightarrow X$ be a surjective morphism with dynamical degree $\delta_f > 1$.  Let $L$ be a big and nef $f$-quasi-polarizable divisor on $X$.  Let $D$ be a non-zero and effective $f$-quasi-polarizable Cartier divisor on $X$ and defined over $\KK$.  Assume that $D$ is linearly equivalent to $mL$ for some $m>0$. For $n \geq 1$, let 
$$D^{(n)} = (f^{(n)})^*D \equiv \delta_f^n m L$$ 
and let $(D^{\operatorname{red}}_{\operatorname{p.i.}})^{(n)}$ be the reduced properly intersecting part of $D^{(n)}$.  Write 
$$(D^{\operatorname{red}}_{\operatorname{p.i.}})^{(n)} = D_1^{(n)}+\dots + D_{q_n}^{(n)}$$ for $D_i^{(n)}$ distinct, irreducible, reduced and properly intersecting Cartier divisors on $X$.  Assume that $D_i^{(n)}$ is linearly equivalent to $m_i^{(n)}L$ for positive integers $m_i^{(n)}$.  

Set 
$$\gamma = \gamma((D^{\operatorname{red}}_{\operatorname{p.i.}})^{(n)}) := \left( \max_i \left\{ m_i^{(n)} \right\}\right) \left( \dim X + 1\right)$$
and
$$
c_n := \frac{\sum_{i=1}^{q_n} m_i^{(n)} - \gamma}{\delta^n_f m} \text{.}
$$

Let $x \in X(\KK)$.  
Then for each finite set $S \subset M_{\KK}$ and all sufficiently small  $\epsilon > 0$, the set 
$$
\left\{ f^{(m)}(x) : \frac{\sum_{v \not \in S} \lambda_{D}(f^{(m)}(x),v) }{m h_L(f^{(m)}(x)) } \leq c_n - \epsilon \right\} \subseteq \mathrm{O}_f(x)
$$
is Zariski-non-dense.  In particular if $c_n > 0$, then the collection of $(D,S)$-integral points in $\mathrm{O}_f(x)$ are non-Zariski dense. 
\end{theorem}

In our formulation of Theorem \ref{Grieve:Noytaptim:fwd:orbits:main:thm} here, we make note of the unfortunate misprint in the statement of \cite[Theorem 1.2]{Grieve:Noytaptim:fwd:orbits} and its proof.  Indeed, for example, the $m^n$ that appears in the definition of the constant $c_n$ there should be replaced by $m$.  Similar considerations apply to the proof.  We thank Zheng Xiao for bringing this to our attention.

Here, motivated by \cite[Theorem 1.18]{Matsuzawa:2023}, we explore what can be said more generally compared to the situation treated in Theorem \ref{Grieve:Noytaptim:fwd:orbits:main:thm} (e.g., without any assumption of quasi-polarizability) and also what can be said more generally compared to the situation treated in \cite{Matsuzawa:2023} (e.g., when $X$ has canonical singularities).  In this regard, we prove the following Theorem \ref{coord:sizes:ZD:main:thm:cor} below.  It is a significantly more general form of \cite[Theorem 1.18]{Matsuzawa:2023} for the case of divisors, on projective varieties with canonical singularities, and defined over finite extensions of the base number field.

\begin{theorem}\label{coord:sizes:ZD:main:thm:cor} Fix a finite set of places $S \subset M_{\KK}$.
Let $f \colon X \rightarrow X$ be a surjective morphism, defined over $\KK$, for $X$ a projective variety over $\KK$ and having canonical singularities.  Let $D$ be a nonzero effective Cartier divisor on $X$ and defined over  some finite extension $\FF / \KK$.  Let $L$ be a big line bundle on $X$ and defined over $\KK$.  Let $x \in X(\KK)$.
\begin{enumerate}
\item
Assume that
\begin{itemize}
\item $\alpha_f(x) > 1$; and
\item $e_f(D) < \alpha_f(x)$ where 
$$e_f(D) := \left( \liminf_{n \to \infty} \left( \inf_E \frac{1}{\operatorname{ord}_E((f^{(n)})^* D)} \right)^{1/n}  \right)^{-1} \text{.}$$ 
\item
Fix $e \in \QQ_{>0}$ with $e_f(D) \leq e < \alpha_f(x)$, fix $\epsilon_0 \in \RR_{>0}$ and $\epsilon \in \QQ_{>0}$ and a positive integer $m_0$ which is such that 
\begin{itemize}
\item[(i)] $e + \epsilon < \alpha_f(x)$; and
\item[(ii)] $(e + \epsilon)^{m_0} < \alpha_f(x)^{m_0} \epsilon_0$.
\end{itemize}
\end{itemize}
\item  Further assume, in addition to the assumptions of item (1), that
that the conclusion of Conjecture \ref{Vojta:Conj:lc:pairs:intro} applied to the case of the log canonical pair $\left(X,\frac{1}{(e+\epsilon)^{m_0}}(f^{(m_0)})^*D\right)$ and $\epsilon' = 1$ holds true.

\end{enumerate}
Then
the set
$$
\left\{ x' \in \mathrm{O}_f(x) : \frac{ \sum_{v \not  \in S} \lambda_{D} (x',v) }{h_L(x') }  \leq \liminf_{n \to \infty} \frac{ \sum_{v \in M_{\KK}} \lambda_{D}(f^{(n)}(x),v)}{h_L(f^{(n)}(x))} - \epsilon \right\}
$$
is not-Zariski dense.

\end{theorem}

We prove Theorem \ref{coord:sizes:ZD:main:thm:cor} in Section \ref{main:thms:proof}.  It is proved in a manner that is similar to the proof of Theorem \ref{coord:sizes:main:thm:cor}. A key point to the proof of Theorems \ref{coord:sizes:main:thm}, \ref{coord:sizes:main:thm:cor} and \ref{coord:sizes:ZD:main:thm:cor} is Proposition \ref{useful:prop} which expands on \cite[Theorem 1.1]{Sano:2018}.

\subsection*{Acknowledgements}  The first author thanks NSERC for their support through his grants DGECR-2021-00218 and RGPIN-2021-03821.  Both authors thank colleagues for their interest and discussions on related topics.  Further, they thank an anonymous referee for careful reading and helpful suggestions.

\section{Absolute values, local Weil and height functions}\label{abs:weil:function:conventions}

Our conventions about absolute values, local Weil and height functions are as in  \cite[Section 2]{Grieve:qualitative:subspace}.  For example, if $X$ is a projective variety over $\KK$ and if $D$ is a Cartier divisor on $X$ and defined over $\FF$ then fixing a \emph{presentation} 
$$\mathcal{D} = (s_D; N, \mathbf{s}; M, \mathbf{t})$$ 
of $D$, and defined over $\FF$, the corresponding local Weil function with respect to a place $v \in M_{\KK}$ for $\mathcal{D}$ and depending on $\mathcal{D}$ is given by
$$
\lambda_{D}(x,v) = \lambda_{\mathcal{D}}(x,v) := \max_{j=0,\dots,k} \min_{i = 0,\dots,\ell} \left|  \frac{s_j}{t_i s_D}(x) \right|_{v,\KK} 
\text{ for $x \in (X \setminus \operatorname{Supp}(D))(\KK)$.  }
$$

Here, we have fixed $w \in M_{\FF}$ with $w \mid v$ and
$$
\Osh_{X_{\FF}}(D) \simeq N \otimes M^{-1}
$$
for line bundles $N$ and $M$ on $X_{\FF}$ had having respective global generating sections 
$$
\mathbf{s} := (s_0,\dots,s_k) \text{ and } \mathbf{t} := (t_0,\dots,t_\ell) \text{;}
$$
$s_D$ is a meromorphic section of $\Osh_{X_{\FF}}(D)$ with associated Cartier divisor equal to $D$.  Finally, $(X \setminus \operatorname{Supp}(D))(\KK)$ is the set of those $\KK$-rational points $x \in X(\KK)$ whose image in $X_{\FF}(\FF)$ does not lie in the support of $D$.  

Local Weil functions for $\QQ$-Cartier divisors on $X$ and defined over $\FF$ are defined in the evident way via $\QQ$-linear extension of the theory of local Weil functions for Cartier divisors.

Further, we denote by $h_L(\cdot)$ the logarithmic height functions on $X$ that are determined by lines bundles $L$ on $X$ and defined over $\KK$.  Especially, if 
$$\mathbf{x} = [x_0:\dots:x_n] \in \PP^n(\KK)$$ 
then
$$
h_{\Osh_{\PP^n}(1)}(\mathbf{x}) := \sum_{v \in M_{\KK}} \max_i |x_i|_v \text{.}
$$
Height functions for $\QQ$-line bundles are defined in an evident way by $\QQ$-linear extension of the theory of height functions for line bundles.

\section{Pairs, log resolutions and discrepancies}\label{birational:geom:conditions}

In our study of Problem \ref{main:problem}, we employ concepts from higher dimensional birational geometry.  For this purpose, here we make precise our conventions about \emph{pairs}, \emph{log resolutions}, \emph{discrepancies} and related notions.   They follow closely those of \cite{Kollar:Mori:1998} (among others).  

For our purposes here a \emph{pair} $(X,D)$ consists of a geometrically normal projective variety $X$ over $\KK$ and a $\QQ$-Weil divisor $D$ on $X$ which is defined over $\FF$ and which is such that $\K_{X_{\FF}} + D$ is $\QQ$-Cartier.  Here $\K_{X_{\FF}}$ is a canonical divisor on $X_{\FF}$.

The concept of \emph{$\FF / \KK$ log resolutions} for a pair $(X,D)$, see Definition \ref{FF:KK:log:res}, is important for our perspective here.  For the case that $D$ is defined over $\KK$, it reduces to the familiar case of a log resolution.  

\begin{defn}\label{FF:KK:log:res}
Let $(X,D)$ be a pair with $X$ defined over $\KK$ and $D$ defined over a finite extension $\FF / \KK$.  By an \emph{$\FF / \KK$ log resolution} for $(X,D)$ we mean a projective birational morphism
$$\mu \colon X' \rightarrow X$$ 
from a nonsingular projective variety $X'$, which is defined over $\KK$, which has the property that  
$$\mu^{-1}_{\FF}(\operatorname{Supp} D) \bigcup \operatorname{Ex}(\mu_{\FF}) \text{,}$$ 
for $\operatorname{Ex}(\mu_{\FF})$ the exceptional set of $\mu_{\FF}$, is a divisor with strict normal crossings support.  
\end{defn}

To place Definition \ref{FF:KK:log:res} into its proper context, note that when $D$ is defined over $\KK$, existence of such log resolutions is a consequence of work of Hironaka \cite{Hironaka:1964}.  (See also \cite{Kollar:Mori:1998} and \cite{Bierstone-et-al}.)  On the other hand, in complete generality, existence of $\FF / \KK$ log resolutions, for example when $(X,D)$ is defined \emph{strictly over $\FF$} in the sense that the pair $(X_{\FF},D)$ does not descend to a pair that is defined over $\KK$, is less clear.  Nonetheless, many pairs $(X,D)$, with $D$ defined over $\FF$ do indeed admit $\FF / \KK$ log resolutions in the sense of Definition \ref{FF:KK:log:res}.  

We refrain from a further discussion here about existence of $\FF / \KK$ log resolutions in the sense that we have formulated in Definition \ref{FF:KK:log:res}.  On the other hand, we do mention that the reason for our interest in such $\FF / \KK$ log resolutions arises from considerations that surround Vojta's Main Conjecture in the form that we formulate in Conjectures \ref{pair:main:conj:intro} and \ref{Vojta:Conj:lc:pairs:intro}.    

Returning to general considerations, consider a pair $(X,D)$, with $X$ defined over $\KK$ and $D$ defined over $\FF$ and suppose given an $\FF/\KK$ log resolution
$$
\mu \colon X' \rightarrow X
$$
in the sense of Definition \ref{FF:KK:log:res}.  Write
$$\K_{X_{\FF}' / X_{\FF}} := \K_{X'_{\FF}} - \mu^* \K_{X_{\FF}} \text{,}$$
$$\K_{X_{\FF}'} + \mu_{\FF *}^{-1}(D) \equiv_{\text{$\QQ$-lin.equiv.}} \mu_{\FF}^*(\K_{X_{\FF}} + D) + \Gamma' $$ 
and 
$$\Gamma' = \sum a_i \Gamma_i'$$ 
for $\Gamma_i'$ are distinct reduced, irreducible exceptional divisors. 

As in \cite[Definition 2.34]{Kollar:Mori:1998}, we say that $(X,D)$ is a \emph{log canonical pair}, or simply that $(X,D)$ is \emph{log canonical}, when the coefficients $a_i$ of $\Gamma_i'$ for each exceptional divisor $\Gamma_i'$ have the property that $a_i \geq -1$.  

We say that a geometrically normal projective variety $X$ is \emph{$\QQ$-Gorenstein} if $\K_X$ is $\QQ$-Cartier.  Further, as in \cite{Kollar:Mori:1998}, see also \cite{Reid:1987}, we say that a $\QQ$-Gorenstein projective variety $X$ has \emph{canonical singularities} if for each resolution of singularities 
$$\mu \colon X' \rightarrow X$$ 
(defined over $\KK$)
we can write
$$
\K_{X'} \equiv_{\text{$\QQ$-lin.equiv.}} \mu^*\K_{X} + \sum_{\substack{
\text{all prime $\mu$-exceptional } \\
\text{
divisors $E_i$
}
}} a_i E_i 
$$
where $a_i \geq 0$ for all $i$.  If $a_i \geq -1$ for all $i$, then $X$ has \emph{log canonical singularities}.
Note, in particular, that when $X$ has canonical singularities $\K_{X'/X}$ is an effective $\QQ$-Cartier divisor.

\section{Multiplier ideal sheaves and log canonical thresholds}\label{lct:multiplier:ideal:sheaves}

Let $X$ be a projective variety over $\KK$ and having log canonical singularities.  Let $\FF / \KK$ be a finite extension and let $\Ish \subset \Osh_{X_{\FF}}$ be a nonzero ideal sheaf.  Let $c \in \QQ_{>0}$ be a positive rational number.  The \emph{multiplier ideal sheaf} $\mathcal{J}(X,\Ish^c)$ associated with $\Ish$ and $c$ is defined by fixing a log resolution
$$
\mu \colon X' \rightarrow X_{\FF}
$$
of $(X,\Ish)$ and setting
$$
\mathcal{J}(X,\Ish^c) := \mu_* \Osh_{X'} ( \K_{X'/X} - \lfloor c F \rfloor) \text{.}
$$
Here, $\mu$ is a projective birational morphism which is such that
\begin{itemize}
\item $X'$ is a nonsingular variety over $\FF$;
\item $\Ish \Osh_{X'} = \Osh_{X'}(-F)$ for some effective Cartier divisor $F$ on $X_{\FF}$;
\item $\operatorname{Exc}(\mu) \bigcup \operatorname{Supp} F$ is a simple normal crossings divisor on $X$ and $\operatorname{Exc}(\mu)$ is the exceptional locus; and
\item $\K_{X'/X} := \K_{X'} - \mu^* \K_{X_{\FF}}$ is the relative canonical divisor.
\end{itemize}

Let $x \in X(\overline{\KK})$.  The log canonical threshold of $\Ish$ at $x$ is defined to be
$$
\operatorname{lct}(\Ish;x) := \inf \{c \in \QQ_{>0} : \mathcal{J}(X,\Ish^c)_x \subset \mathfrak{m}_x  \} \text{.}
$$
The log canonical threshold of $(X,\Ish)$ is defined to be
$$
\operatorname(X,\Ish) := \min \{\operatorname{lct}(\Ish;x) : x \in X(\overline{\KK}) \} \text{.}
$$

\begin{example}\label{divisior:lct}
Consider the case that $\Ish = \Osh_{X_{\FF}}(-D)$ is the ideal sheaf of a nonzero effective Cartier divisor $D$ on $X$ and defined over $\FF$.  In this case, we denote
$$
\operatorname{lct}(D) = \operatorname{lct}(X,D) := \operatorname{lct}(X,\Osh_{X_{\FF}}(-D)) \text{;}
$$
then 
$$
\operatorname{lct}(D) = \operatorname{lct}(X,D)
$$
can be described, as in \cite[p.~74]{Cheltsov:Shramov:2008} and \cite[Equation 13]{He:Ru:2022}, as 
\begin{align*}
\operatorname{lct}(D) = \operatorname{lct}(X,D) & = \sup \{ c \in \QQ : (X,cD) \text{ is log canonical}\}  \\
& = \inf_E \frac{1 + a(E,X)}{\operatorname{ord}_E(D)}
\end{align*}
where the infimum is taken over all exceptional divisors $E$ over $X_{\overline{\KK}}$ and supported on some projective model $X' \rightarrow X_{\overline{\KK}}$ and where
$$a(E,X) := \ord_E(\K_{X'/X_{\overline{\KK}}})$$
is the \emph{discrepancy} of $E$ over $X$.  For later use, we note in particular, that if $X$ has canonical singularities then 
\begin{equation}\label{lct:div:inequal:eqn}
\operatorname{lct}(D) \geq \inf_E \frac{1}{\operatorname{ord}_E(D)} \text{.}
\end{equation}
\end{example}

\begin{defn}\label{mult:defn}
When $X$ has canonical singularities if $f \colon X \rightarrow X$ is a surjective morphism and if $D$ is a nonzero effective Cartier divisor on $X$ and defined over a finite extension $\FF / \KK$, we define
$$
e_f(D) := \left( \liminf_{n \to \infty} \left( \inf_E \frac{1}{\operatorname{ord}_E((f^{(n)})^* D)} \right)^{1/n}  \right)^{-1} \text{.}
$$
Here, the infimum is taken over all exceptional divisors $E$ over $X_{\overline{\KK}}$.
\end{defn}

The motivation for Definition \ref{mult:defn} arises from \cite[Lemma 7.11 and Corollary 7.13]{Matsuzawa:2023}.  We refer to Examples \ref{Matsuzawa:mult:eg} and \ref{extended:Matsuzawa:mult:eg} for a more detailed discussion.  Here, since we are working with divisors, as opposed to higher codimensional subschemes, we are able to take a more direct approach to capture, in a more robust setting, the essence of the asymptotical dynamical measures of multiplicities that are considered there.  

For later use, in the proof of our main results, we make note of an evident observation in Remark \ref{dynamic:lc:rmk} below.

\begin{remark}\label{dynamic:lc:rmk}
In the setting of Definition \ref{mult:defn}, observe that if $e \in \QQ_{>0}$ is such that $e_f(D) \leq e$ and if $\epsilon > 0$, then there exists $m_0 \geq 1$ which is such that 
$$
\operatorname{lct}((f^{(m)})^* D) \geq \frac{1}{(e+\epsilon)^m} \text{ for all $m \geq m_0$.}
$$
This observation is important to our proof of Theorems \ref{coord:sizes:main:thm:cor} and \ref{coord:sizes:ZD:main:thm:cor}.  We refer to Examples \ref{Matsuzawa:mult:eg} and \ref{extended:Matsuzawa:mult:eg} for a more detailed discussion.
\end{remark}

\section{Arithmetic inequalities for log canonical pairs}\label{vojta:lc:pairs:sec}

In the direction of Problem \ref{main:problem},  we derive instances of its conclusion as consequences of certain Diophantine arithmetic  inequalities for log canonical pairs.  (See Conjectures \ref{pair:main:conj:intro}, \ref{Vojta:Conj:lc:pairs:intro} and Theorem \ref{coord:sizes:main:thm}.)  These inequalities are in the spirit of Vojta's Main Conjecture.

Recall that Conjectures \ref{Vojta:Conj:lc:pairs:intro} and  \ref{pair:main:conj:intro} are formulated in Section \ref{intro}.  Conjecture \ref{pair:main:conj:intro} is a more traditional form of Vojta's Main Conjecture for normal crossings pairs and $\KK$-rational points. (Compare for example with \cite[Conjecture 14.3.2]{Bombieri:Gubler}.)   However, our formulation here is still a more general form of the traditional formulation of the conjecture.  Indeed, here we allow the normal crossings divisor to have field of definition some finite extension of the given base number field.  Developing the theory in this way, in the presence of a given field extension $\FF / \KK$, is significant for applications and is in line with the perspective of \cite{Grieve:Divisorial:Instab:Vojta}, \cite{Grieve:points:bounded:degree}, \cite{Grieve:qualitative:subspace} (among others).  The purpose of this section is to make several remarks that surround Conjectures \ref{Vojta:Conj:lc:pairs:intro} and  \ref{pair:main:conj:intro}.

In Proposition \ref{Vojta:log:res:implies:Vojta:lc:pairs} below, starting with Conjecture \ref{pair:main:conj:intro}, we show that it implies Conjecture \ref{Vojta:Conj:lc:pairs:intro}, which is  a form of \emph{Vojta's Main Conjecture for log canonical pairs}.  First, in
Examples \ref{Yasuda:Eg:Conj}, \ref{SchmidtsSubSpaceThm:Eg} and \ref{Ding:destab:eg} below, we indicate the manner in which a special case of Conjecture \ref{Vojta:Conj:lc:pairs:intro} fits into the Diophantine arithmetic inequalities that have been proposed in \cite[Example 5.3]{Yasuda:2018}.

\begin{example}\label{Yasuda:Eg:Conj}
Consider the case that $X$ is nonsingular and that $D$ is a simple normal crossings divisor on $X$ and defined over $\KK$.  Then $(X,D)$ is log canonical and, as in  \cite[Example 5.3]{Yasuda:2018}, the conclusion of Conjecture \ref{Vojta:Conj:lc:pairs:intro} takes the form of \cite[Conjecture 5.2]{Yasuda:2018}. 
\end{example}

\begin{example}\label{SchmidtsSubSpaceThm:Eg}
The traditional form of Schmidt's Subspace Theorem for $\PP^n$ and linear forms with coefficients in $\FF$, see for example \cite{Grieve:qualitative:subspace} or \cite{Bombieri:Gubler}, gives the most basic unconditional instance of the conclusion of Conjecture \ref{Vojta:Conj:lc:pairs:intro}.  
\end{example}

\begin{example}\label{Ding:destab:eg}
Suppose that $X$ is a $\QQ$-Fano variety over $\KK$ and that $X$ is \emph{Ding $\K$-destabilized}, compare for instance with \cite[Theorem 1.2]{Fujita:2018:AJM}, by a properly intersecting log canonical pair $(X',B)$ in the sense that the inequality 
$$
1 \leq \operatorname{lct}(X;B_i) \leq \frac{ \int_0^{\infty} \Vol_{X'_{\FF}}(- \mu^* \K_{X_{\FF}} - t B_i) \mathrm{d} t }{ \Vol_{X}(-\K_X) } 
$$
holds true.
Here 
$$\mu \colon X' \rightarrow X$$ is a proper model of $X$, defined over $\KK$, 
$$B = \sum_{i = 1}^q B_i$$ 
is a nonzero effective Cartier divisor on $X'_{\FF}$ with the property that its irreducible components $B_1,\dots,B_q$ intersect properly and $\Vol(\cdot)$ is the \emph{volume function} \cite[Definition 2.2.31]{Laz}.  In this context, the conclusion of Conjecture \ref{Vojta:Conj:lc:pairs:intro} follows as a consequence of the Ru-Vojta arithmetic general theorem (compare for example with \cite[Corollary 1.3]{Grieve:2018:autissier}, \cite[Corollary 1.13]{Ru:Vojta:2016}, \cite[Theorem 1.1]{Grieve:Divisorial:Instab:Vojta} or \cite[Theorem 1.2]{Grieve:chow:approx}).
\end{example}

As additional evidence towards Conjecture \ref{Vojta:Conj:lc:pairs:intro}, our first result shows that it is implied by the conclusion of Vojta's conjecture applied to an $\FF / \KK$ log resolution of $(X',D')$ (assuming that such an $\FF / \KK$ log resolution indeed exists).  This is formulated in more precise terms in the following way.  Proposition \ref{Vojta:log:res:implies:Vojta:lc:pairs} complements \cite[Proposition 5.4]{Yasuda:2018}. 

\begin{proposition}\label{Vojta:log:res:implies:Vojta:lc:pairs}
Let $X$ be a projective variety over $\KK$ and having canonical singularities. Let $D'$ be a nonzero effective effective $\QQ$-Cartier divisor over $X$, defined over $\FF$ and supported as a $\QQ$-Cartier divisor on $X'_{\FF}$ for  some proper model 
$$\mu' \colon X' \rightarrow X$$ 
of $X$ and defined over $\KK$. 

Assume that $(X',D')$ is log canonical and that $(X',D')$ admits an $\FF / \KK$ a log resolution 
$$
\mu'' \colon X'' \rightarrow X'
$$
in the sense of Definition \ref{FF:KK:log:res}.  

Write 
$$
\K_{X''_{\FF}}  + \Gamma'' = (\mu''_{\FF})^*(\K_{X'_{\FF}}+ D')
$$
where 
$$\Gamma'' = \sum_i b_i \Gamma_i''$$
with $b_i \leq 1$  and $\Gamma_i''$ are distinct reduced and irreducible divisors. 
Set 
$$
\Gamma_{\operatorname{red}}'' = \sum_i \Gamma_i'' \text{.}
$$

Fix a finite set of places $S \subset M_{\KK}$.  Let $\epsilon > 0$ and 
$$L' = (\mu')^* L$$ 
for $L$ a big line bundle on $X$.  

Then the conclusion of Vojta's Conjecture, in the form of Conjecture \ref{pair:main:conj:intro}, applied to $(X'', \Gamma_{\operatorname{red}}'')$ with respect to 
$$L'' = (\mu'')^* L'$$ 
and for $\epsilon > 0$ implies the conclusion of Conjecture \ref{Vojta:Conj:lc:pairs:intro} for $(X',D')$ with respect to $L'$ and $\epsilon' = \epsilon > 0$.
\end{proposition}

\begin{proof}
Write 
$$
\Gamma'' = (\mu''_{\FF})^*(\K_{X'_{\FF}} + D') - \K_{X''_{\FF}} = \sum b_i E_i'' \text{.}
$$
Then $\Gamma''$ has strict normal crossings support and has coefficients $b_i \leq 1$.  

So, in particular, $\Gamma_{\operatorname{red}}''$ is a strict normal crossings divisors and further 
$$ \Gamma'' \leq \Gamma_{\operatorname{red}}'' \text{.}$$

Now, by the conclusion of Vojta's Conjecture, in the form of Conjecture \ref{pair:main:conj:intro}, applied to $\Gamma_{\operatorname{red}}''$ on $X''$ with respect to $L''$ and $\epsilon' > 0$, there is a proper Zariski closed subset $$Z'' \subsetneq X''$$ 
which is defined over $\KK$ and which is such that 
\begin{multline}\label{prop:proof:eqn1}
\sum_{v \in S} \lambda_{\Gamma''}(\cdot,v) + h_{\K_{X''}}(\cdot)  \\
\leq \sum_{v \in S} \lambda_{\Gamma_{\operatorname{red}}''}(\cdot,v) + h_{\K_{X''}}(\cdot)  \leq 
\epsilon h_{L''}(\cdot) + \mathrm{O}(1) \text{ on $X'' \setminus Z''$.}
\end{multline}
In what follows at times we implicitly enlarge $Z''$ if needed.

Now, since
$$
\Gamma'' = (\mu_{\FF}'')^*(\K_{X'_{\FF}} + D') - \K_{X''_{\FF}}
$$
we deduce from Equation \eqref{prop:proof:eqn1} that
\begin{multline}\label{prop:proof:eqn2}
\sum_{v \in S} \lambda_{(\mu''_{\FF})^*(D')} (\cdot,v) + \sum_{v \in S} \lambda_{\mu''^*\K_{X'}}(\cdot,v) - \sum_{v \in S} \lambda_{\K_{X''}}(\cdot,v) + h_{\K_{X''}}(\cdot) \\ \leq \epsilon h_{L''}(\cdot) + \mathrm{O}(1) \text{ over $X' \setminus Z'$.}
\end{multline}
Further, since
$$\K_{X'' / X'} = \K_{X''} - (\mu'')^* \K_{X'}$$ 
is a $\QQ$-Cartier divisor, Equation \eqref{prop:proof:eqn2} can be rewritten in the form
\begin{multline}\label{prop:proof:eqn3}
\sum_{v \in S} \lambda_{(\mu''_{\FF})^*(D')}(\cdot,v) - \sum_{v \in S} \lambda_{\K_{X''/X'}}(\cdot,v) + h_{\K_{X''}}(\cdot) \\ \leq \epsilon h_{L''}(\cdot) + \mathrm{O}(1) \text{ over $X' \setminus Z'$.}
\end{multline}
Therefore
\begin{multline}\label{prop:proof:eqn4}
\sum_{v \in S} \lambda_{(\mu''_{\FF})^*(D')}(\cdot,v) \\ \leq \epsilon h_{L''}(\cdot) - h_{\K_{X''}}(\cdot) + \sum_{v \in S} \lambda_{\K_{X''/X'}}(\cdot,v) + \mathrm{O}(1) \text{ over $X' \setminus Z'$.}
\end{multline}
Since 
$$\K_{X'' / X'} = \K_{X''} - (\mu'')^* \K_{X'} \geq 0 \text{ and } \K_{X'} \geq (\mu')^* \K_X$$ 
it follows using Equation \eqref{prop:proof:eqn4} that
\begin{multline}\label{prop:proof:eqn5}
 \sum_{v \in S} \lambda_{(\mu''_{\FF})^*(D')}(\cdot,v) 
 \leq  \epsilon h_{L''}(\cdot) - h_{\K_{X''}}(\cdot) + h_{\K_{X''/X'}}(\cdot) + \mathrm{O}(1) \\ 
 \leq \epsilon h_{L''}(\cdot) - h_{(\mu'')^* \K_{X'}}(\cdot) + \mathrm{O}(1) \text{ over $X' \setminus Z'$.}
\end{multline}
Therefore, in light of Equation \eqref{prop:proof:eqn5}, the conclusion is that 
$$
\sum_{v \in S} \lambda_{D'}(\cdot,v) \leq \epsilon h_{L'}(\cdot) - h_{\mu^* \K_{X'}}(\cdot) + \mathrm{O}(1) \text{ over $X' \setminus Z'$}
$$
for some proper Zariski closed subset $Z' \subsetneq X'$.
\end{proof}

\section{Preliminaries about arithmetic dynamical degrees}
Our main result towards Problem \ref{main:problem} is Theorem \ref{coord:sizes:main:thm}.  Its formulation utilizes the concept of  arithmetic dynamical degree from \cite{Kawaguchi:Silverman:2016}.  Here, we fix the most basic preliminary details in this regard which are important for our purposes here.

Let $X$ be a geometrically normal projective variety over a number field $\KK$.    Let $h_H(\cdot)$ be a logarithmic height function on $X$ that is associated to an ample line bundle $H$ on $X$ and defined over $\KK$.  Set 
$$
h^+_H(\cdot) := \max \{ h_H(\cdot),1 \} \text{.}
$$

Let $f \colon X \rightarrow X$ be a surjective $\KK$-morphism.  Then the \emph{arithmetic degree} of $x \in X(\overline{\KK})$ with respect to $f$ is the well-defined quantity (i.e., independent of $H$)
$$
\alpha_f(x) = \lim_{n \to \infty} h_H^+(f^{(n)}(x))^{1/n} \text{.}
$$
Here $f^{(n)} \colon X \rightarrow X$ is the $n$th interate of $f$.

In our study of Problem \ref{main:problem} a key role is played by the following more flexible form of \cite[Proposition 1.5]{Matsuzawa:2023}.  

\begin{proposition}\label{useful:prop}  Let $L$ be a big line bundle on a geometrically normal projective variety $X$ and defined over $\KK$.   Use $\QQ$-linear equivalence to write $L \sim_{\QQ} A + N$ for $A$ ample and $N$ effective.  Let $f \colon X \rightarrow X$ be a surjective morphism. Let $x \in X(\KK)$.   There exists a proper Zariski closed subset $Z \subsetneq X$ together with a non-negative integer $\ell$ and positive real numbers $C_1$, $C_2$ which are such that
$$
C_1 n^{\ell} \alpha_f(x)^n \leq \max \{1, h_L(f^{(n)}(x))\} \leq C_2 n^{\ell} \alpha_f(x)^n 
$$
for all $n \geq 1$ with the property that $f^{(n)}(x) \not \in \operatorname{Supp}(N)$.  Here $\alpha_f(x)$ is the arithmetic degree of $x$ with respect to $f$.
\end{proposition}
\begin{proof}
The first step is to deduce the case that $L$ is ample.  This goes through exactly as in \cite[Proof of Theorem 1.1]{Sano:2018}.  
To boost this to the case that $L$ is big, we using $\QQ$-linear equivalence to write $L \sim_{\QQ} A + N$ for $A$ ample and $N$ effective, we then deduce the result outside of the support of $N$.
\end{proof}

\begin{corollary}\label{useful:prop:cor}
Let $L$ be a big line bundle on a geometrically normal projective variety $X$ and defined over $\KK$.   Use $\QQ$-linear equivalence to write $L \sim_{\QQ} A + N$ for $A$ ample and $N$ effective.   Let $f \colon X \rightarrow X$ be a surjective morphism.  Let $x \in X(\KK)$.  Then there exists constants $C_1, C_2 > 0$ and $\ell \geq 0$ which are such that for all pairs of positive integers $n \geq m \geq 0$, it holds true that
$$
\frac{h_L(f^{(n-m)}(x))}{h_L(f^{(n)}(x))} \leq \frac{C_2(n-m)^\ell \alpha_f(x)^{n-m}}{C_1 n^\ell \alpha_f(x)^n} = \frac{C_2(n-m)^\ell}{C_1 n^\ell \alpha_f(x)^m} 
$$ 
 for $f^{(n-m)}(x), f^{(n)}(x) \not \in \operatorname{Supp}(N)$.
\end{corollary}

\begin{proof}
This is deduced easily from Proposition \ref{useful:prop}.
\end{proof}

\section{Asymptotic dynamics of coordinate sizes along divisors}\label{coord:sizes:main:thm:sec}

Theorem \ref{coord:sizes:main:thm} below is our main result towards Problem \ref{main:problem}.  It complements \cite[Theorem 1.16]{Matsuzawa:2023}  and should be seen as a more general instance of that result.  As an application of Theorem \ref{coord:sizes:main:thm}, we use it to prove Theorem \ref{coord:sizes:main:thm:cor}.  Theorem \ref{coord:sizes:ZD:main:thm:cor} is proved using similar techniques.

\begin{theorem}\label{coord:sizes:main:thm}
Let $X$ be a geometrically normal projective variety and defined over $\KK$.  Assume that $X$ has canonical singularities.  Let $L$ be a big line bundle on $X$ and defined over $\KK$.  Let $D$ be a nonzero and effective $\QQ$-Cartier divisor on $X$.  We allow the possibility that $D$ is defined over some finite extension $\FF / \KK$.  Let $f \colon X \rightarrow X$ be a surjective morphism.  Let $x \in X(\KK)$.  Further
\begin{itemize}
\item{assume that the forward orbit $\mathrm{O}_f(x)$ is \emph{generic} in the sense that the forward orbit $\mathrm{O}_f(x)$ has finite intersection with all proper Zariski closed subsets of $X$;}
\item{assume that $\alpha_f(x) > 1$;}
\item{
choose $e \in \QQ_{> 0} $ so that $0 < e < \alpha_f(x)$;
}
\item{
let $\epsilon_0 > 0$  be a positive real number;
}
\item{
let $\epsilon > 0$ be a rational number such that $e + \epsilon < \alpha_f(x)$ and fix $m_0$ and $m \geq m_0$ so that 
$$
\frac{(e + \epsilon)^m}{\alpha_f(x)^m} < \epsilon_0 \text{; and}
$$
}
\item{
assume that $\left(X,\frac{1}{(e+\epsilon)^m}(f^{(m)})^*D\right)$ is log-canonical for some $m \geq m_0$ and that the conclusion of Conjecture \ref{Vojta:Conj:lc:pairs:intro} applied to $\left(X,\frac{1}{(e+\epsilon)^m}(f^{(m)})^*D\right)$ and $\epsilon' = 1$ holds true.
}
\end{itemize}

Then
$$
\lim_{n \to \infty} \frac{ \sum_{v \in S} \lambda_{D}(f^{(n)}(x),v)}{h_L(f^{(n)}(x))} = 0 \text{.}
$$

\end{theorem}

\begin{proof}

Working away from a proper Zariski closed subset of $X$ we may and do assume that $h_L(\cdot) \geq 1$.  

Then by Conjecture \ref{Vojta:Conj:lc:pairs:intro} applied to $\left(X,\frac{1}{(e+\epsilon)^m}(f^{(m)})^*D\right)$ and $\epsilon' = 1$, there exists a Zariski closed subset $Z \subsetneq X$ which is such that
\begin{multline}\label{dynamical:vojta:lct:inequalities:cor:eqn1}
\sum_{v \in S} \lambda_{D}(f^{(n)}(x), v) \leq \sum_{v \in S} \lambda_{(f^{(m)})^*D}(f^{(n-m)}(x), v) + \mathrm{O}(1) \\
\leq \epsilon_0 \alpha_f(x)^m(h_L(f^{(n-m)}(x)) - h_{\K_X}(f^{(n-m)}(x))) + \mathrm{O}_m(1) 
\end{multline}
for all $n \geq m$ with $f^{(n-m)}(x) \not \in Z$.   

Now recall that by assumption, the forward orbit $\mathrm{O}_f(x)$ is generic.   Thus 
$$f^{(n-m)}(x) \not \in Z$$ 
for all $n \geq n_0$ for some sufficiently large $n_0$.   Therefore enlarging $Z$ if required, applying Corollary \ref{useful:prop:cor}, there exist constants $C_1, C_2 > 0$ and $\ell \geq 0$ such that 
$$
\frac{h_L(f^{(n-m)}(x))}{h_L(f^{(n)}(x))} \leq \frac{C_2(n-m)^\ell}{C_1 n^\ell \alpha_f(x)^m} 
$$
for all $n \geq n_0$ and $f^{(n-m)}(x)$ with $f^{(n)}(x) \not \in \operatorname{Supp}(N)$.

Also, working away from $\operatorname{Supp}(N)$ if needed, take a constant $C_3 > 0$ which is such that 
$$h_L(\cdot) - h_{\K_X}(\cdot) \leq C_3 h_L(\cdot) \text{.}$$

Then for all $n \geq n_0$, it follows upon dividing \eqref{dynamical:vojta:lct:inequalities:cor:eqn1} by $h_L(f^{(n)}(x))$ that
\begin{multline*}
\frac{\sum_{v \in S} \lambda_{D}(f^{(n)}(x),v)}{h_L(f^{(n)}(x))} \leq \epsilon_0 \alpha_f(x)^m \frac{C_3 h_L(f^{(n-m)}(x))}{h_L(f^{(n)}(x))} + \frac{\mathrm{O}_m(1)}{h_L(f^{(n)}(x))} \\
\leq \epsilon_0 \alpha_f(x)^m \frac{C_3 C_2(n-m)^\ell \alpha_f(x)^{n-m}}{C_1 n^\ell \alpha_f(x)^n} + \frac{\mathrm{O}_m(1)}{h_L(f^{(n)}(x))} \\
\leq \frac{C_2 C_3}{C_1} \epsilon_0 + \frac{\mathrm{O}_m(1)}{h_L(f^{(n)}(x))} \text{.}
\end{multline*}

Finally, we note that
\begin{itemize}
\item{$C_1$, $C_2$ and $C_3$ are independent of $\epsilon_0$ and $n$;}
\item{
$\mathrm{O}_m(1)$ is independent of $n$; and
}
\item{$h_L(f^{(n)}(x)) \to \infty$ as $n \to \infty$.
}
\end{itemize}

The desired conclusion then follows.
\end{proof}

In Example \ref{Matsuzawa:mult:eg} below, we explain how \cite[Theorem 1.16]{Matsuzawa:2023} fits within the the context of Theorem \ref{coord:sizes:main:thm} and the more general framework that we are considering here.

\begin{example}[Compare with \cite{Matsuzawa:2023}]\label{Matsuzawa:mult:eg}  
Consider the case of a surjective morphism 
$f \colon X \rightarrow X$
for $X$ a projective variety over $\KK$.  Recall, compare with \cite[Definition 4.1]{Matsuzawa:2023}, that the \emph{multiplicity} of $f$ along a scheme point $x \in X$ can defined to be
$$
e_f(x) := \ell_{\Osh_{X,x}}(\Osh_{X,x} / f^* \mathfrak{m}_{f(x)} \Osh_{X,x}) \text{.}
$$
Further, following \cite[Theorem 4.8]{Matsuzawa:2023}, we set
$$
e_{f,-}(x) := e_-(x) := \lim_{n \to \infty} \left( \sup \left\{ e_{f^{(n)}}(y) : y \in X, f^{(n)}(y) = x \right\}\right)^{1/n} \text{.}
$$

Let $Y \subseteq X$ be a proper closed subscheme.  Set
$$
e := \max \{e_{f,-}(x) : x \in Y \} \text{.}
$$
When $X$ is nonsingular, it follows from \cite[Corollary 7.13]{Matsuzawa:2023} that  for all $\epsilon > 0$ there is a $m_0 \geq 1$ which is such that
$$
\operatorname{lct}(X,(f^{(m)})^{-1}(Y)) \geq \frac{1}{(e+\epsilon)^m}
$$
for all $m \geq m_0$.  

Here $\operatorname{lct}(X,(f^{(m)})^{-1}(Y))$ is the \emph{log canonical threshold} of $(X,(f^{(m)})^{-1}(Y))$ (as defined in \cite{Matsuzawa:2023} and using the framework of multiplier ideal sheaves; compare also with our approach in Section \ref{lct:multiplier:ideal:sheaves}) for $(f^{(m)})^{-1}(Y)$ the scheme theoretic inverse image of $Y$ under the $m$th iterate of $f$.   Thus, in this setting the hypothesis of Theorem \ref{coord:sizes:main:thm} reduces to the hypothesis of \cite[Theorem 1.16]{Matsuzawa:2023}. 
\end{example}

\begin{example}\label{extended:Matsuzawa:mult:eg}
In a spirit that is similar to Example \ref{Matsuzawa:mult:eg}, consider the case of a surjective morphism $f \colon X \rightarrow X$ 
for $X$ a projective variety over $\KK$ and having canonical singularities.  Let $D$ be a nonzero Cartier divisor on $X$ and defined over a finite extension field $\FF / \KK$.  Then upon setting 
$$e := e_f(D) = \left( \liminf_{n \to \infty} \left( \inf_E \frac{1}{\operatorname{ord}_E((f^{(n)})^* D)} \right)^{1/n}  \right)^{-1} \text{,} $$ 
the conclusion is, as noted in Remark \ref{dynamic:lc:rmk}, that  for all $\epsilon > 0$,
there exists $m_0 \geq 1$ which is such that 
$$
\operatorname{lct}((f^{(m)})^*D) \geq \frac{1}{(e+\epsilon)^m} \text{ for all $m \geq m_0$.}
$$
\end{example}

\section{Proof of Theorems \ref{coord:sizes:main:thm:cor} and \ref{coord:sizes:ZD:main:thm:cor}}\label{main:thms:proof}

Here, we prove Theorems \ref{coord:sizes:main:thm:cor} and \ref{coord:sizes:ZD:main:thm:cor}.  To keep matters in perspective, as mentioned in Section \ref{intro}, Theorem \ref{coord:sizes:main:thm} implies Theorem \ref{coord:sizes:main:thm:cor}.  Among other features, Theorem \ref{coord:sizes:main:thm:cor} below, extends \cite[Theorem 1.16]{Matsuzawa:2023} to the case of big line bundles on varieties with canonical singularities and nonzero effective divisors with coefficients in some finite extension of the base number field.

We also note that when $L = -\K_X$ is big, the conclusions of Theorem \ref{coord:sizes:main:thm:cor} and \ref{coord:sizes:ZD:main:thm:cor} hold true unconditionally provided that $\left(X, \frac{1}{(e + \epsilon)^{m_0}} (f^{(m_0)})^* D \right)$ is a properly intersecting Ding $\K$-destabilizing pair in the sense of Example \ref{Ding:destab:eg}.

\begin{proof}[Proof of Theorem \ref{coord:sizes:main:thm:cor}]
As in Example \ref{extended:Matsuzawa:mult:eg}, we note that the hypothesis item (1) of Theorem \ref{coord:sizes:main:thm} is satisfied.  The hyposthesis of item (2) is also evidently satisfied.  Thus, Theorem \ref{coord:sizes:main:thm:cor} follows as a consequence of Theorem \ref{coord:sizes:main:thm}.
\end{proof}

\begin{proof}[Proof of Theorem \ref{coord:sizes:ZD:main:thm:cor}]
As in Example \ref{extended:Matsuzawa:mult:eg}, we note that the hypothesis item (1) of Theorem \ref{coord:sizes:ZD:main:thm:cor} is  satisfied.  The proof then proceeds along the lines of \cite[Proof Theorem 1.18]{Matsuzawa:2023}.

Indeed, without loss of generality, we may and do assume that $h_L(\cdot) \geq 1$.  Then, by assumption $e < \alpha_f(x)$ and so we can choose $\epsilon \in \QQ_{>0}$ which is such that 
$$
e + \epsilon < \alpha_f(x) \text{.}
$$
Now let $\epsilon_0>0$ and fix a suitable $m_0 \geq 1$ which is such that 
$$
\frac{ (e + \epsilon)^{m_0} }{ \alpha_f(x)^{m_0} } < \epsilon_0
$$
and which is such that the pair $\left(X, \frac{1}{(e + \epsilon)^{m_0}} (f^{(m_0)})^* D \right)$ is log canonical.  The existence of such an $m_0$ follows as in Example \ref{extended:Matsuzawa:mult:eg}.  (See also Remark \ref{dynamic:lc:rmk}.)

By the conclusion of Conjecture \ref{Vojta:Conj:lc:pairs:intro} applied to $\left(X, \frac{1}{(e + \epsilon)^{m_0}} (f^{(m_0)})^* D \right)$ and with respect to $\epsilon' = 1$, it follows that there is a proper Zariski closed subset $Z \subsetneq X$ which is such that 
\begin{multline}\label{my:eqn:1}
\sum_{v \in S}  \lambda_{D}(f^{(n)}(x),v) \\ \leq \epsilon_0 \alpha_f(x)^{m_0}( h_L(f^{(n-m_0)}(x)) - h_{\K_X}(f^{(n-m_0)}(x)) )  + \mathrm{O}_{m_0}(1)
\end{multline}
for all $n \geq m_0$ with $f^{(n-m_0)}(x) \not \in Z$.  In what follows, we may and do assume that $f^{(n)}(x) \not \in \operatorname{Supp}(D)$.  Therefore we conclude from \eqref{my:eqn:1} that 
\begin{multline}\label{my:eqn:2}
\sum_{v \not \in S}  \lambda_{D}(f^{(n)}(x),v) \\ 
\geq \sum_{v \in M_{\KK}} \lambda_{D}(f^{(n)}(x),v) - \epsilon_0 \alpha_f(x)^{m_0}(h_L(f^{(n-m_0)}(x)) 
 - h_{\K_X}(f^{(n-m_0)}(x)) - \mathrm{O}_{m_0}(1)
\end{multline}
if 
$$f^{(n)}(x) \not \in Z' := f^{(m_0)}(Z) \bigcup \operatorname{Supp}(D) \text{.}$$

Enlarging $Z'$ if needed and fixing constants $C_1, C_2 > 0$ and $\ell \geq 0$ which are such that 
\begin{equation}\label{my:eqn:3}
C_1 n^\ell \alpha_f(x)^n \leq h_L(f^{(n)}(x)) \leq C_2 n^\ell \alpha_f(x)^n
\end{equation}
for all $n \geq 1$ and fixing a constant $C_3 > 0$ which is such that
$$
h_L(\cdot)  - h_{\K_X}(\cdot) \leq C_3 h_L(\cdot)
$$
the conclusion, upon combining Equations \eqref{my:eqn:2} and \eqref{my:eqn:3}, is that
$$
\frac{\sum_{v \not \in S}  \lambda_{D}(f^{(n)}(x), v) }{ h_L(f^{(n)}(x)) } \geq \frac{ \sum_{v \in M_{\KK} } { \lambda_{D}(f^{(n)}(x),v)} }{ h_L(f^{(n)}(x)) } - \epsilon_0 \frac{C_2 C_3}{C_1} - \frac{\mathrm{O}_{ m_0}(1)}{h_L(f^{(n)}(x))}
$$
if $f^{(n)}(x) \not \in Z'$.

The desired conclusion then follows since $C_1$, $C_2$, $C_3$ are independent of $\epsilon_0$, $n$; and since $\mathrm{O}_{m_0}(1)$ is independent of $n$ and $h_L(f^{(n)}(x)) \to \infty$.
\end{proof}

\providecommand{\bysame}{\leavevmode\hbox to3em{\hrulefill}\thinspace}
\providecommand{\MR}{\relax\ifhmode\unskip\space\fi MR }
\providecommand{\MRhref}[2]{%
  \href{http://www.ams.org/mathscinet-getitem?mr=#1}{#2}
}
\providecommand{\href}[2]{#2}

\end{document}